\documentclass{amsart}
\usepackage{amsmath}
\usepackage{bbm}
\usepackage{latexsym}
\usepackage{xcolor}
\usepackage{tikz}
\usetikzlibrary{arrows}
\usetikzlibrary{patterns}
\usetikzlibrary{shapes}
\usepackage{mathrsfs}
\usepackage{bbm}
\newcommand{\R}{\mathbb{R}}

\newcommand{\N}{\mathbb N}

 \usepackage[usenames,dvipsnames]{pstricks}
 \usepackage{epsfig}
 \usepackage{pst-grad} 
 \usepackage{pst-plot} 

\newtheorem{thm}{Theorem}
\newtheorem{lemma}[thm]{Lemma}
\newtheorem{cor}[thm]{Corollary}
\newtheorem{prop}[thm]{Proposition}

\newtheorem*{thmi}{Theorem}
\newtheorem*{cori}{Corollary}

\theoremstyle{remark}
\newtheorem{example}[]{Example}
\newtheorem{remark}[]{Remark}

\newcommand{\be}{\begin{equation}}
\newcommand{\ee}{\end{equation}}

\usepackage{mathtools} 
\mathtoolsset{showonlyrefs,showmanualtags} 

\title[Homotopy properties of endpoint maps]
{Homotopy properties of endpoint maps and a theorem of Serre in subriemannian geometry}\author{Francesco Boarotto}
\author{Antonio Lerario}

\begin{document}
\begin{abstract}We discuss homotopy properties of endpoint maps for affine control systems. We prove that these maps are Hurewicz fibrations with respect to some $W^{1,p}$ topology on the space of trajectories, for a certain $p>1$.  We study critical points of geometric costs for these affine control systems, proving that if the base manifold is compact then the number of their critical points is infinite (we use Lusternik-Schnirelmann category combined with the Hurewicz property). In the special case where the control system is \emph{subriemannian} this result can be read as the corresponding version of Serre's theorem, on the existence of infinitely many geodesics between two points on a compact riemannian manifold. In the subriemannian case we show that the Hurewicz property holds for all $p\geq1$ and the horizontal-loop space with the $W^{1,2}$ topology has the homotopy type of a CW-complex (as long as the endpoint map has at least one regular value); in particular the inclusion of the horizontal-loop space in the ordinary one is a homotopy equivalence.\end{abstract}
\maketitle

\section{Introduction}
In this paper we study homotopy properties of the set of those curves on a manifold $M$ whose velocities are constrained in a nonholonomic way (these curves are called \emph{horizontal}). The nonholonomic constraint is made explicit by requiring that the curves should be tangent to a totally nonintegrable distribution (for example a contact distribution, whose horizontal curves are called \emph{legendrian}). More generally we will allow \emph{affine} constraints, by considering a set of vector fields $\mathcal{F}=\{X_0, X_1, \ldots, X_d\}$ and defining a horizontal curve $\gamma:I=[0,1]\to M$  to be an \emph{absolutely continuous} curve (hence differentiable almost everywhere) solving the equation:
\be\label{eq:control} \dot\gamma=X_0(\gamma)+\sum_{i=1}^du_iX_i(\gamma), \quad \gamma(0)=x\ee
for functions $u_1, \ldots, u_d$ called \emph{controls} ($x\in M$ is a point that we fix from the very beginning). 

The vector field $X_0$ is special (it plays the role of a ``drift'') and in many interesting cases, like the subriemannian, it is assumed to be zero; the remaining vector fields satisfy the totally nonintegrable \emph{H\"ormander} condition: a finite number of their iterated brackets should span the whole tangent space $TM$ (this is also called the \emph{bracket generating} condition).

The regularity we impose on the controls determines the topology on the space $\Omega$ of all horizontal curves (called also \emph{trajectories}). In this paper we will assume $u=(u_1, \ldots, u_d)\in L^{p}(I, \R^d)$ for some $1< p<\infty$ (thus we consider the $W^{1,p}$ topology on the space of trajectories). The correspondence between a curve and its controls defines local coordinates on $\Omega$, which in turn becomes a Banach manifold modeled on $L^p=L^p(I, \R^d)$ (in fact this manifold is just the open subset of $L^p$ consisting of all controls whose corresponding trajectory is defined on the all interval $I$, see the Appendix of this paper or \cite{Montgomery} for more details); as a byproduct of this identification we will often replace a curve with the $d$-tuple of controls describing it in local coordinates. 

The \emph{endpoint map} is the map that associates to each trajectory its final point:
\be F:\Omega \to M\quad \gamma\mapsto \gamma(1).\ee
This map is differentiable (smooth in the $W^{1,2}$ case \cite{AgrachevBarilariBoscain}), and the set:
\be\Omega(y)=F^{-1}(y)\ee
with the induced topology coincides with the set of horizontal curves joining $x$ to $y$. 
 
In the riemannian case, these spaces are well understood and their topological properties are related to those of the manifold $M$ via the \emph{path fibration} (see \cite{BottTu, Hatcher}), which in our setting we discuss below. 

The uniform convergence topology  on $\Omega$ has been studied in \cite{Sarychev} and the $W^{1,1}$ in \cite{dynamic}. For the scopes of calculus of variations the case $W^{1,p}$ with $p>1$ is especially interesting as the analysis becomes more pleasant: for example the $p$-th power of the $L^p$ norm becomes a $C^1$ function and one can apply classical techniques from critical point theory to many problems of interest. Also, it is worth recalling that already in the subriemannian case not all topologies on $\Omega$ are equivalent a priori: for example in the $W^{1,\infty}$ case the so-called \emph{rigidity} phenomenon appear: some curves might be isolated (up to reparametrization) in the $W^{1,\infty}$ topology \cite{Bryant}.

The key property for studying the topology of horizontal path spaces is the homotopy lifting property for the endpoint map. Our first result generalizes the main results from \cite{dynamic, Sarychev}, proving that there exists $p_c>1$ (depending on $\mathcal{F}$) such that endpoint map is a \emph{Hurewicz fibration} for the $W^{1, p}$ topology for all $1\leq p<p_c$ (i.e. $F$ has the homotopy lifting property with respect to any space for these topologies).
\begin{figure}\scalebox{1} 
{
\begin{pspicture}(0,-0.6280469)(14.02291,0.6280469)
\psline[linewidth=0.04cm](0.9210156,-0.030351562)(12.321015,-0.010351563)
\psdots[dotsize=0.14, fillstyle=solid,dotstyle=o](0.9210156,-0.030351562)
\psdots[dotsize=0.14](12.321015,-0.010351563)
\psdots[dotsize=0.14](8.921016,-0.010351563)
\psdots[dotsize=0.14,fillstyle=solid,dotstyle=o](5.5410156,-0.030351562)
\usefont{T1}{ptm}{m}{n}
\rput(1.1024708,0.43464842){$W^{1,\infty}$}
\usefont{T1}{ptm}{m}{n}
\usefont{T1}{ptm}{m}{n}
\rput(5.702471,0.41464844){$W^{1,p}$}
\usefont{T1}{ptm}{m}{n}
\rput(12.402471,0.41464844){$L^{\infty}$}
\usefont{T1}{ptm}{m}{n}
\rput(12.31247,-0.40535155){$\textrm{(Sarychev)}$}
\usefont{T1}{ptm}{m}{n}
\rput(8.902471,-0.40535155){$\textrm{(Dominy and Rabitz)}$}
\usefont{T1}{ptm}{m}{n}
\rput(9.10247,0.43464842){$W^{1,1}$}
\usefont{T1}{ptm}{m}{n}
\rput(5.5224705,-0.38535157){$(p_c>1)$}
\end{pspicture} 
}
\caption{A picture of the continuous inclusions (from left to right) of the various $W^{1,p}([0,1])$ spaces.}\label{fig:inclusions} 
\end{figure}
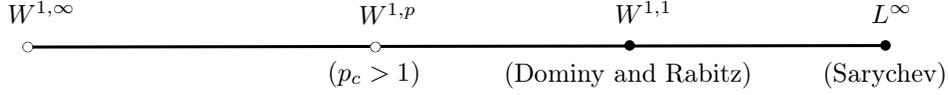
\begin{thmi}[The endpoint map is a Hurewicz fibration]\label{thmi:sarychev}There exists an interval $[1, p_c)\subseteq [1, \infty)$ (depending on $\mathcal{F}$), such that if $p\in [1, p_c)$ the Endpoint map $F:\Omega\to M$ is a Hurewicz fibration for the $W^{1,p}$ topology on $\Omega$. Moreover if $X_0=0$ then $p_c=\infty$. \end{thmi}
It is remarkable that the subriemannian case $(X_0=0)$ has the Hurewicz fibration property for all $1\leq p<\infty$, as in general if $X_0\neq 0$ the endpoint map can fail to have the homotopy lifting property for some finite $p<\infty$, as shown in the next example \cite{AgrachevLee}. 
\begin{example}\label{ex:AgrachevLee}Consider $M=\R^2$ with coordinates $(x_1, x_2)$ and:
\be X_0=x_1^2\partial_{x_2},\quad X_1=\partial_{x_1},\quad X_2=x_1^k\partial_{x_2}, \quad k\geq 3\ee
We consider on $\Omega$ the function $u\mapsto J(u)=\|u\|_2^2$, which is continuous for every $p\geq 2$ for the $W^{1,p}$ topology. Let us also consider the function $c_1:\R^2\times \R^2\to \R$:
\be c_1(x,y)=\inf\{J(\gamma)\in \Omega\,|\,\gamma(0)=x, \, \gamma(1)=y\}.\ee
In \cite[Proposition 2.1]{AgrachevLee} it is proved that there exists $K>0$ such that for all $w\in \R$ and all $z<0$:
\be c_1((0,w), (0, w))=0\quad \textrm{and}\quad c_1((0, w), (0, z))\geq K.\ee
Consider now the path $g_s=(0, -s)$ and let $u_0\in \Omega$ be a lift for $g_0$ (i.e. $F(u_0)=g_0$). Now this path (a homotopy of inclusions of a single point) cannot be lifted: an existence of such a lift would be a continuous path $u_s$ on $\Omega$ with $u_s\in \Omega(g_s)$, and in particular:
\be\lim_{s\to 0}J(u_s)=0,\ee
which contradicts the fact that $J|_{\Omega(g_s)}\geq K>0$ for all $s>0$.
\end{example}

Our proof of the previous theorem is much inspired from \cite{Sarychev, dynamic} and in fact consists in a simple (but important) modification of the proof from \cite{Sarychev}. This theorem has two immediate consequences for the $W^{1,p}$ topology ($1<p<p_c$): (i) all the spaces $\Omega(y)$ as $y$ varies on $M$ are homotopy equivalent; (ii) the inclusion of $\Omega(y)$ into the ordinary space of curves (with no nonholonomic constraints) is a weak homotopy equivalence. In particular this last property allows to adapt what is known for the topology of the standard loop space to our horizontal one. 

\begin{cori}[Some topological implications]\label{cori:lust}For every $k\in \mathbb{N}$ , every $1\leq p<p_c$ and every $y\in M$ the following isomorphism between homotopy groups holds for the $W^{1,p}$ topology:
\be \pi_k(\Omega(y))\simeq \pi_{k+1}(M).\ee
Moreover if the base manifold $M$ is compact and simply connected, then the Lusternik-Schnirelmann category of the space $\Omega(y)$ is infinite.
\end{cori}

Once there is some information available for the topology of $\Omega(y)$, it can be used to study critical points of functionals, the classical example being the study of geodesics between two points. A celebrated theorem of Serre \cite{Serre} states that if a riemannian manifold $M$ is compact, then every two points are joined by infinitely many geodesics; the proof of this theorem essentially uses the topology of $\Omega(y)$ to force the existence of critical points of the Energy functional, which in the riemannian case are exactly geodesics. 

More generally one can study critical points of the $p$-Energy $J_p:u\mapsto \|u\|_p^p$ on $\Omega(y)$ for affine control systems on \emph{regular fibers} $\Omega(y)$: as long as $1<p<p_c$ this function is $C^1$ (Lemma \ref{lemma:C1}) and when restricted to $\Omega(y)$ it satisfies the Palais-Smale condition (Proposition \ref{prop:PS}). These two properties allow to use classical results to force the existence of critical points.

\begin{thmi}[On the critical points for the $p$-Energy]
	Let $y$ be a regular value for the endpoint map of the control system \eqref{eq:control}, $1<p<p_c$ and consider $f=J_p|_{\Omega(y)}$. If the base manifold $M$ is compact then $f$ has infinitely many critical points.
\end{thmi}

As a corollary, we thus obtain a subriemannian version of the Serre's theorem: given $x$ and any regular point $y$ for the endpoint map on a compact subriemannian manifold there are infinitely many geodesics connecting them. In some cases (e.g. contact or fat distributions) the assumption of $y$ being a regular value may be dropped: in these situations there are no abnormal curves other than the trivial ones, and our arguments are essentially not affected.

It is clear at this point that deciding whether a given subriemannian manifold possesses or not at least a regular value for $F$ is a crucial problem and, to our knowledge, still an open question. If such a value exists, we can be more precise about the homtopy type of the fibers.

\begin{thmi}[The homotopy type of the fiber]
	Assume that the endpoint map has at least a regular value on $M$. Then any fiber $\Omega(y)$ endowed with the $W^{1,2}$ topology has the homotopy type of a $CW$-complex.
\end{thmi}

\subsection{Related work}The problem of understading the topology of the space of maps with some restrictions on their differential goes back to the works on immersions of S. Smale \cite{Smale1}, for the case of curves on a manifold the author considers spherical-type constraints on the velocities (i.e. immersions and regular homotopies). Hurewicz properties for endpoint maps of affine control systems were studied first by A. V.  Sarychev \cite{Sarychev} for the uniform convergence topology and by J. Dominy and H. Rabitz \cite{dynamic} for the $W^{1,1}$ topology. 
The quantitative study of the interaction between the topology of the horizontal loop space and the set of geodesics was initiated by the second author together with A. Agrachev and A. Gentile in \cite{AgrachevGentileLerario}. In the contact case a ``local'' version of Serre's theorem was investigated by the second author and L. Rizzi in \cite{LerarioRizzi} (the authors perform an asymptptic count of the number of geodesics between two point on a contractible contact manifold, using the relation between a subriemannian manifold and its nilpotent approximation).

\subsection{Structure of the paper}Section \ref{sec:preliminaries} is devoted to the proof of the Hurewicz fibration property (Theorem \ref{thm:sarychev}): the crucial ingredient is the construction of a continous cross-section for the endpoint map (Proposition \ref{prop:cross}). The topological implications are discussed in Section \ref{sec:implications}. In section \ref{sec:critical} we study critical points of geometric costs: the Palais-Smale property is proved in Proposition \ref{prop:PS} and applications via Lusternik-Schnirelmann method are discussed in Section \ref{sec:LSappli}. The subriemannian case is discussed in Section \ref{sec:subriemannian}. The Appendix contains some additional technical results, mostly known to experts.
\subsection*{Acknowledgements} We wish to thank E. Le Donne for bringing the problem to our attention, as well as A. A. Agrachev, D. Barilari and L. Rizzi for stimulating discussions. The second author also thanks M. Degiovanni for interesting discussions and for bringing his attention to critical points techniques in Banach spaces. Part of this research was done during the trimester ``Geometry, Analysis and Dynamics on Sub-Riemannian manifolds`` at IHP, Paris: we wish to thank the organizers of the trimester for the wonderful working atmosphere. The second author has received funding from the European Community’s Seventh Framework Programme ([FP7/2007-2013] [FP7/2007-2011]) under grant agreement no [258204].
\section{Homotopy properties of the endpoint map}
\subsection{Some preliminary results}\label{sec:preliminaries}

\begin{lemma}\label{lemma:phi}Let $0<\beta<\frac{p}{p-1}$ and  for every $j=1, \ldots, N$ define the map $\rho_j:\R^N\to L^{p}([0, \infty))$ by:
 $\rho_j(r)=0$ if $r_j=0$ and  $\rho_j(r)=\chi_{j}r_j|r_j|^{-\beta}$ otherwise ($\chi_j$ is the characteristic function of the interval $[|r_{j-1}|^{\beta}, |r_{j-1}|^\beta+|r_{j}|^\beta]$ and $r_0=0$). Then the map $\rho_j$ is continuous. \end{lemma}
\begin{proof}
The only needed verification is continuity at zero:
\begin{align} \lim_{r_j\to 0}\|\chi_{j}r_j|r_j|^{-\beta}\|_p&=\lim_{r_j\to 0}\left(\int_{|r_{j-1}|^\beta}^{|r_{j-1}|^\beta+|r_j|^\beta}\left|r_j|r_j|^{-\beta}\right|^pdt\right)^{1/p}\\&=\lim_{r_j\to 0}|r_j|^{\frac{\beta+p-\beta p}{p}}=0
 \end{align}since $\beta+p-\beta p>0.$
\end{proof}

\begin{prop}[The cross-section]\label{prop:cross}Given the manifold $M$ and the family of vector fields $\mathcal{F}$, there exists an interval $[1, p_c)\subset [1, \infty)$ such that for every $1\leq p<p_c$ every point in $M$ has a neighborhood $W$ and a continuous map:
\begin{align}\hat{\sigma}: W\times W&\to L^{p}([0, \infty), \R^d)\times \R \\
(x,y)&\mapsto (\sigma(x,y), T(x,y)) \end{align}
such that  
$F^{T(x,y)}_x(\sigma(x,y))=y$ and $\hat{\sigma}(x,x)=(0,0)$ for every $x,y\in W$. Moreover, if $X_0=0$ then $p_c=\infty$.\end{prop}
\begin{proof}
We first work out the case $X_0=0$ and $p>1$ (the case $p=1$ and $X_0=0$ is a special case of \cite[Lemma 1]{dynamic}, whose notation we follow closely). Given the vector fields $\{Y_1, \ldots, Y_k\}$ define inductively  $Q^1(Y_1)=e^{Y_1}$ and:
\be Q^{\nu}(Y_1, \ldots, Y_\nu)=e^{Y_\nu}\circ Q^{\nu-1}(Y_1, \ldots, Y_{\nu-1})\circ e^{-Y_\nu}\circ(Q^{\nu-1}(Y_1, \ldots, Y_{\nu-1}))^{-1},\quad \nu\geq1.\ee
Given a real number $r$ we define also:
\be P^{\nu}(Y_1, \ldots, Y_{\nu}, r)=Q^{\nu}(rY_1, \ldots, rY_\nu).\ee
It follows from the Baker-Campbell-Hausdorff formula that, for $r$ sufficiently small, \be P^{\nu}(Y_1, \ldots, Y_{\nu}, r^{1/\nu})=e^{r\,\textrm{ad}Y_{\nu}\dotso\textrm{ad}Y_2Y_1+\textrm{higher order terms in}\,r}.\ee
Then the bracket generating condition on $\mathcal{F}$ implies that (see \cite[Section 2.1]{Jean} or the proof of \cite[Lemma 1]{dynamic}) every point in $M$ has a neighborhood $W$ and a continuous\footnote{The $k$-th component of $\phi=(\phi_1, \ldots, \phi_n)$ is the $\nu$-th root of a $C^1$ function.} map $\phi:W\times W\to \R^n$ such that $\phi(x,x)=0$ for all $x\in W$ and:
\be\label{eq:prod} \left(\prod_{k=1}^{n} P^{\nu_k}( X_{k_1}, \ldots, X_{k_{\nu_k}}, \phi_k(x,y))  \right)(x)=y \quad \forall x,y\in W.\ee
Now we notice that the product in \eqref{eq:prod} can be written as:
 \be\left(\prod_{k=1}^{n} P^{\nu_k}(X_{k_1}, \ldots, X_{k_{\nu_k}}, \phi_k(x,y)) \right)=\prod_{j=1}^Ne^{\phi_{a_j}(x,y) X_{b_j}} \ee
 where $N$ is a given number and $a_j, b_j\in \{1, \ldots, d\}$ for $j=1, \ldots, N$ (these numbers are fixed and depend on the neighborhood $W$ only).
 
 Given $p>1$ choose $\beta$ satisfying the hypothesis of Lemma \ref{lemma:phi}. Using the notation of Lemma \ref{lemma:phi}  we can now interpret $y=(\prod_{j=1}^Ne^{\phi_{k_j}(x,y) X_{k_j}})(x)$ as the solution at time:
 \be T(x,y)=\sum_{j=1}^N|\phi_{k_j}(x,y)|^\beta\ee
 of the control problem with initial datum $y(0)=x$ and control:
 \be \sigma(x,y)=\left(\sum_{\{j\,|\,k_j=1\}} \rho_j(\phi_{k_j}(x,y)), \sum_{\{j\,|\,k_j=2\}} \rho_j(\phi_{k_j}(x,y)), \ldots,\sum_{\{j\,|\,k_j=d\}} \rho_j(\phi_{k_j}(x,y))\right). \ee
 By Lemma \ref{lemma:phi} it follows that the map $\hat{\sigma}=(\sigma, T)$ defined in this way is continuous: each component is the sum of compositions of continuous functions ($T(x,y)$ is continuous since $\beta>0$) and $\hat{\sigma}(x,x)=(0,0).$
 
For the case $X_0\neq 0$ we notice that the proof of \cite[Lemma 1]{dynamic} produces indeed the continuity of the cross section for some $1<p<p_c$ (as we will see, a lower bound for $p_c$ in this case is given by $\sigma/(\sigma-1)$, where $\sigma$ is the step of the distribution $\mathcal{F}$). We simply check the needed details. The sequence of exponentials \eqref{eq:prod} now has to be replaced with \cite[Equation 6.a]{dynamic} (using the same notation as the mentioned paper):
\be\label{eq:proddyn} \left(\prod_{k=1}^{n} R^{\nu_k}(X_0,  X_{k_1}, \ldots, X_{k_{\nu_k}}, \pm \phi_{k_j}(x,y), \phi_{k_j}(x,y), \ldots, \phi_{k_j}(x,y))\right).\ee
 The construction in \cite{dynamic} works in such a way that given $\alpha>\nu_k/2$, using BCH formula, $R^{\nu_k}$ can be written as the exponential of a series of terms from $\{\phi_{k_1}^{2\alpha}X_0, \ldots, \phi_{k_\nu}^{2\alpha}X_0,\phi_{k_1} X_{k_1}, \ldots, \phi_{k_\nu} X_{k_\nu}\}$ and their Lie brackets. We choose thus $\alpha>\sigma/2$ which guarantees $\alpha> \nu_k/2$ for all $k=1, \ldots, n.$ The product in \eqref{eq:proddyn} can thus be regarded as the solution at time $T=\sum_j \nu_k\phi_{k_j}^{2\alpha}$ of a control problem with initial datum $y(0)=x$ and locally constant controls $\sigma=(\sigma_1, \ldots, \sigma_d)$ taking values on an interval of length $\phi_{k_j}^{2\alpha}$. The continuity of the final time $T$ follows from the fact that $\alpha>0$; for the continuity of the corresponding $\sigma$ we argue as in \cite[Appendix C]{dynamic}. Each component of $\sigma$ is the concatenation of some fixed number of locally constant controls (some of them can possibly be zero) each one defined on an interval of length $\phi_{k_j}^{2\alpha}$ and taking a value proportional to $\phi_{k_j}^{1-2\alpha}$. Then it is enough to check the continuity of this control at zero for the $L^p$-topology. If we choose $p<\frac{2\alpha}{2\alpha-1}$ then:
 \be \lim_{\phi_{k_j}\to0}\int_{c}^{c+\phi_{k_j}^{2\alpha}} \left|\phi_{k_j}^{1-2\alpha}\right|^pdt=\lim_{\phi_{k_j}\to0}\phi_{k_j}^{2\alpha+p-2p\alpha}=0. \ee 
 (Notice in particular that, because of the way we chose $\alpha$, a lower bound for $p_c$ is given by $\sigma/(\sigma-1)$.)
 \end{proof}
 
		\begin{prop}[Rescaled concatenation]\label{prop:rescaling}Let $p\in [1,\infty)$, then the map $\mathcal C: L^p(I)\times L^p([0,+\infty))\times \R\to L^p(I)$ defined below is continuous:
		\[
		\mathcal C(u,v,T)(t)=\left\{
		\begin{alignedat}{9}
		& (T+1)u(t(T+1))\quad&& 0\leq t< \frac{1}{T+1}\\
		& (T+1)v((T+1)t-1),\quad&& \frac{1}{T+1}< t\leq 1.
		\end{alignedat}
		\right.
		\]	
	Moreover (extending the definition componentwise to controls with value in $\R^d$) we also have $F_x^{1+T}(u*v)=F_x^1(\mathcal{C}(u,v,t))$ for every $x\in M$ (here $u*v$ denotes the usual concatenation).
	 \end{prop}
	\begin{proof}
		As $L^p(I)\times L^p([0,+\infty))\times \R$ is a metric space, it is sufficient to prove that if $(u_k,v_k,T_k)\to (u,v,T)$, then $\|\mathcal C(u_k,v_k,T_k)-\mathcal C(u,v,T)\|_p\to 0$.
		
		Assume for simplicity that $T_k\geq T$ (we can split the sequence $\{T_k\}_{k\in \mathbb{N}}$ into two monotone subsequences and work the case $T_k\leq T$ separately, it is completely analogous). Start with:
		\begin{align}
		\label{eqn:1c}
		  &\|\mathcal C(u_k,v_k,T_k)-\mathcal C(u,v,T)\|^p_p\\
		  &=\int_{0}^{1/(T_k+1)}|(T_k+1)u_k(t(T_k+1))-(T+1)u(t(T+1))|^pdt\\
		  &+\int_{1/(T_k+1)}^{1/(T+1)}|(T_k+1)v_k(t(T_k+1)-1)-(T+1)u(t(T+1))|^pdt\\
		  &+\int_{1/(1+T)}^1|(T_k+1)v_k(t(T_k+1)-1)-(T+1)v(t(T+1)-1)|^pdt.
		\end{align}
	
		Fix $\varepsilon>0$ and let $g$ be a smooth function compactly supported on $[0,3/2)$ such that $\|g-u\|_p\leq \varepsilon$. Observe that for $k$ sufficiently large we have $\|u_k-g\|_p\leq \|u-u_k\|_p+\varepsilon\leq 2\varepsilon$. 
		We can bound the first integral in \eqref{eqn:1c} as:
		\begin{alignat}{9}
		\label{eqn:2c}
		  &\int_{0}^{1/(T_k+1)}&&|(T_k+1)u_k(t(T_k+1))-(T+1)u(t(T+1))|^pdt\\&\leq 2^{2(p-1)}\bigg(&&\int_0^{1/(T_k+1)}|(T_k+1)u_k(t(T_k+1))-(T_k+1)g(t(T_k+1))|^pdt+\\
		  &&&\int_0^{1/(T_k+1)}|(T_k+1)g(t(T_k+1))-(T+1)g(t(T+1))|^pdt+\\
		  &&&\int_0^{1/(T_k+1)}|(T+1)g(t(T+1))-(T+1)u(t(T+1))|^pdt\bigg)\\
		  &\leq 2^{2(p-1)}\bigg(&&|T_k+1|^{p-1}\|u_k-g\|_p^p+|T+1|^{p-1}\|u-g\|_p^p+\\
		  &&&\int_0^{1/(T_k+1)}|(T_k+1)g(t(T_k+1))-(T+1)g(t(T+1))|^pdt\bigg).
		\end{alignat}
		Since $g$ is uniformly continuous in $[0,1]$, the last integral in \eqref{eqn:2c} can also be made as small as we wish as $k\to\infty$ as it is evident from:
		\begin{alignat}{9}
		\label{eqn:3c}
		 &\int_0^{1/(T_k+1)}&&|(T_k+1)g(t(T_k+1))-(T+1)g(t(T+1))|^pdt\\
		 &\leq 2^{p-1}\bigg(&&\int_0^{1/(T_k+1)}|(T_k+1)g(t(T_k+1))-(T_k+1)g(t(T+1))|^p+\\
		 &&&\int_0^{1/(T_k+1)}|(T_k+1)g(t(T+1))-(T+1)g(t(T+1))|^p\bigg).
		\end{alignat}
		
		The third integral in \eqref{eqn:1c} is formally the same as the one just handled; a similar reasoning proves that it goes to zero as $k\to\infty$. We are left to deal with the middle one. In this case as $k\to \infty$ by the dominated convergence theorem we have both
		\[
		  \int_{1/(T_k+1)}^{1/(T+1)}|v_k(t(T_k+1)-1)|^pdt=|T_k+1|^{p-1}\int_0^{(T_k+1)/(T+1)}|v(z)|^pdz\to 0
		\]
		and
		\[
		  \int_{1/(T_k+1)}^{1/(T+1)}|(T+1)u(t(T+1))|^pdt=|T+1|^{p-1}\int_{(T+1)/(T_k+1)}^1|u(z)|^pdz\to 0.
		\]
	Finally this yields:
		\begin{align}
		\label{eqn:4c}
		&\int_{1/(T_k+1)}^{1/(T+1)}|(T_k+1)v_k(t(T_k+1)-1)-(T+1)u(t(T+1))|^pdt\\
		&\leq2^{p-1}\left(|T_k+1|^{p-1}\int_0^{(T_k-t)/(T+1)}|v(z)|^pdz+|T+1|^{p-1}\int_{(T+1)/(T_k+1)}^1|u(z)|^pdz\right),
		\end{align}
		and with this we can eventually conclude that:
		\be\lim_{k\to\infty}\|\mathcal C(u_k,v_k,T_k)-\mathcal C(u,v,T)\|^p_p=0.\ee
		\end{proof}

\subsection{The Hurewicz fibration property and its consequences}\label{sec:implications}

\begin{thm}\label{thm:sarychev}There exists an interval\footnote{Depending on $(M, X_0, X_1, \ldots, X_d)$.} $[1, p_c)\subseteq [1, \infty)$, such that if $p\in [1, p_c)$ the Endpoint map $F:\Omega\to M$ is a Hurewicz fibration for the $W^{1,p}$ topology on $\Omega$. Moreover if $X_0=0$ then $p_c=\infty$. \end{thm}
\begin{proof}
Recall that \emph{Hurewicz fibration} means that $F$ has the homotopy lifting property with respect to every space $Z$.
By Hurewicz uniformization theorem \cite{Hurewicz}, it is enough to show that the homotopy lift property holds locally, i.e. every point $x\in M$ has a neighborhood $W$ such that $F|_{F^{-1}(W)}$ has the homotopy lifting property with respect to any space.

The case $p=1$ is proved in \cite{dynamic}, thus let $1<p<p_c$, $W$ and $\hat{\sigma}$ be given as in Proposition \ref{prop:cross}. Consider a continuous map $g:Z\times I\to W$ and a lift $\tilde{g}_0:Z\to \Omega$ such that $F(\tilde{g}_0(z))=g(z, 0)$ for all $z\in Z.$ We define the lifting homotopy $\tilde{g}:Z\times I\to \Omega$ by:
\be \tilde{g}(z,s)=\mathcal{C}(\tilde{g}_0(z),\underbrace{\sigma(g(z,0), g(z,s)), T(g(z,0), g(z,s))}_{\hat{\sigma}(g(z,0), g(z,s))})\ee
(here $\mathcal{C}$ is defined as in Proposition \ref{prop:rescaling} componentwise).

The defined function $\tilde{g}$ is the composition of continuous functions (by Propositions \ref{prop:cross} and \ref{prop:rescaling}). Moreover by the second assertions in Propositions \ref{prop:cross} and \ref{prop:rescaling}:
\be F(\tilde{g}(z,s))=g(z,s)\quad \forall (z,s)\in Z\times I,\ee
which proves the claim.


\end{proof}
\begin{remark}[On the homotopy type of the fibers]As a consequence of Theorem \ref{thm:sarychev} all fibers of $F$ (even the singular fibers) have the same homotopy type \cite{Spanier}. Moreover, by the long exact homotopy sequence of Hurewicz fibrations \cite{Spanier} one also obtains the following isomorphisms between homotopy groups:
\be\label{eq:pik} \pi_k(\Omega(y))\simeq \pi_{k+1}(M)\quad \forall k\geq0\ee
 \end{remark}
\begin{cor}\label{cor:lust}If the base manifold $M$ is compact and simply connected, then for every $p<p_c$ (where $p_c$ is given by Theorem \ref{thm:sarychev}) and  every $y\in M$ the Lusternik-Schnirelmann category of the space $\Omega(y)$ with respect to the $W^{1,p}$ topology is infinite.
\end{cor}
\begin{proof}
Let $1<p<p_c$ be given by Theorem \ref{thm:sarychev}. Then the ordinary loop space and $\Omega(y)$ are \emph{weakly} homotopy equivalent (both spaces endowed with the $W^{1,p}$-topology). Moreover the Endpoint map for the $W^{1,p}$-\emph{ordinary} loop space is a Hurewicz fibration for every $p>1$ (it is a submersion), in particular the \emph{ordinary} loop spaces are all \emph{weakly} homotopy equivalent to the one with the $W^{1,2}$-topology. Since the cup length of the $W^{1,2}$-\emph{ordinary} loop space of a compact simply connected manifold is infinite (see \cite[Corollary 20]{Schwartz} or the classical work of Serre \cite{Serre}), so it is for $\Omega(y)$ with the $W^{1,p}$-topology. The cup-length is a lower bound for the Lusternik-Schnirelmann category, hence the result follows.
\end{proof}

\section{Critical points of geometric costs}\label{sec:critical}

\subsection{The regularity of the Energy}
For $p>1$ we define the $p$-Energy $J_p:L^p(I,\R^d)\to \R$ by (for simplicity we omit to make explicit the dependence of $J_p$ on $p$, when it will be clear from the context):
\be J_p(u)=\sum_{i=1}^d\|u_i\|_p^p, \quad u=(u_1, \ldots, u_d).\ee
To simplify notations below we will simply denote $L^p=L^p(I, \R^d)$, also we will omit the subscript notation for $u=(u_1, \ldots, u_d)$ when not needed (the corresponding equations should thus be interpreted componentwise).
 
We will need the following result on Nemitski operators.
\begin{thm}[Theorem 2.2 \cite{AmbrosettiProdi}]\label{thm:Nemitski}Let $g:I\times \R\to \R$ be a function such that (i) the function $v\mapsto g(t,v)$ is continuous for almost every $t\in I$; (ii) the function $t\mapsto g(t,v)$ is measurable for all $v\in \R$. Assume also there exists $a,b>0$ such that:
\be |g(t,v)|\leq a+b|v|^\alpha,\quad \alpha=\frac{p}{q}.\ee
Then the map $u(\cdot)\mapsto g(\cdot,u(\cdot))$ (a \emph{Nemitski operator}) is continuous from $L^p(I)$ to $L^q(I).$
\end{thm}
As a corollary we derive the following elementary lemma.
 \begin{lemma}\label{lemma:C1}The map $u\mapsto u|u|^{p-2}$ is a continuous map from $L^{p}(I)$ to $L^{\frac{p}{p-1}}(I)$. In particular, if $y$ is a regular value of the Endpoint map, then $f=J\big|_{\Omega(y)}$ is a $C^1$ function.
\end{lemma}
\begin{proof}The continuity of $u\mapsto u|u|^{p-2}$ is immediate from the previous Theorem. Now, if $y$ is a regular value of the Endpoint, the differential $d_uf$ coincides with $d_uJ|_{T_u\Omega(y)}$ thus to prove that it is differentiable with continuous derivative it is enough to prove it for $J$. The differential $d_uJ$ as a linear functional on $L^p(I, \R^d)$ is easily computed to be (componentwise):
\be \langle d_uJ, h\rangle =\int_{0}^1pu(t)|u(t)|^{p-2}h(t)dt,\quad \textrm{for all $h\in L^p$},\ee
i.e. $d_uJ=pu|u|^{p-2}\in L^q=(L^p)^*$, then the result is clear from the previous claim.
\end{proof}

\begin{prop}[Palais-Smale condition]\label{prop:PS}Let $y$ be a regular value of the Endpoint map and $p>1$. Then the function $f=J|_{\Omega(y)}$  satisfies the Palais-Smale condition, i.e. any sequence $\{u_k\}_{k\in \mathbb{N}}\subset \Omega(y)$ on which $f$ is bounded and such that $d_{u_{k}}f\to 0$ has a convergent subsequence.
\end{prop}
\begin{proof}

Consider the differential $d_uF$ of the endpoint map at a point $u$. Using the notations of Theorem \ref{thm:differentiability} we can write it, for any $v\in L^p$ as: 
\be (d_uF)v=\int_0^1M_u(1)M_u(s)^{-1}B_u(s)v(s)ds.\ee

Denote by $w_1(t;u), \ldots, w_n(t;u)$ the rows of the matrix $M_u(1)M_u(t)^{-1}B_u(t)$; notice that for $j=1, \ldots, d$ we have $w_j(\cdot;u)\in L^q$. If $u\in \Omega(y),$ then we can write:
    	\[
    	 T_u\Omega(y)=\ker d_uF=\textrm{span}\{w_1(\cdot; u), \ldots, w_n(\cdot;u)\}^{\perp};
    	\]
    	as the latter is a linear subspace, we also deduce:
    	\[
    	  T_u\Omega(y)^\perp=\textrm{span}\{w_1(\cdot; u), \ldots, w_n(\cdot;u)\}.
    	\]
    	In particular, for any $u\in\Omega(y)$, $T_u\Omega(y)$ is a closed subspace of codimension $n$ in $L^p$ and therefore it is complemented, i.e. there exists a closed and finite dimensional subspace $W_u$ such that
    	\begin{equation}
    	  \label{eqn:1}
    	  L^p=T_u\Omega(y)\oplus W_u;
    	\end{equation}
    	finally, observe that there exists a continuous linear projection $\pi_u:L^p\to W_u$ subordinated to this splitting, that is $\ker(\pi_u)=T_u\Omega(y)$, see \cite[Chapter 2]{Carothers}.
    	
    	\smallskip
    	
    	Let now $\{u_k\}_{k\in\N}\subset \Omega(y)$ be a bounded sequence such that $d_{u_k}f\to 0$. Since $d_uf=(d_uJ)|_{T_{u}\Omega(y)}$ then by definition of the projections $\pi_{u_k}$ we have:
    	\[
    	  \langle d_{u_k} J, (\textrm{Id}-\pi_{u_k})v\rangle\to 0,\quad \forall v\in L^p.
    	\]
    	The space $L^p$ is uniformly convex, hence reflexive by the Milman-Pettis theorem; the sequence $\{u_k\}$ is bounded by assumption and invoking Banach-Alaoglu we deduce the existence of a subsequence $\{u_{k_l}\}_{l\in\N}$ and $\overline u\in L^p$ such that $u_{k_l}\rightharpoonup \overline u$.  Furthermore, observe that if $q=p^*=\frac{p}{p-1}$ is the conjugate exponent of $p$, then:
    	\[
    	  d_uJ=pu|u|^{p-2}\Rightarrow \|d_uJ\|_q^q=\|u\|_p^p.
    	\]

    	By the above discussion, up to subsequences, we may thus assume that $\|u_k\|_p<C$ and $u_k\rightharpoonup \overline u$ in $L^p$. There exists then $K\in\N$ sufficiently large so that for any norm-one $v\in L^p$ and $k>K$ the following holds:
    	\begin{equation}
    	\label{eqn:2a}
    	  |\langle d_{u_k}J,\pi_{u_k}(v)\rangle|\leq|\langle d_{u_k}J,v\rangle|+|\langle d_{u_k}J ,v-\pi_{u_k}(v)\rangle|<C+1.
    	\end{equation}
    	It is well-known \cite[Section 3]{Carothers} that the splitting in \eqref{eqn:1} induces a dual splitting on $L^q$, namely for any $u\in\Omega(y)$ we have
    	\[
    	  L^q=(T_u\Omega(y))^*\oplus W_{u}^*;
    	\] 
        moreover the adjoint operator $\pi_{u_k}^*$ is still a projection with kernel $W_{u_k}^\perp$ and range $(T_{u_k}\Omega(y))^\perp \cong W_{u_k}^*\cong L^q/W_{u_k}^\perp=\textrm{span}\{w_1(\cdot; u_k), \ldots, w_n(\cdot;u_k)\}.$ In particular, \eqref{eqn:2a} shows that 
        \[
          \|\pi_{u_k}^*(d_{u_k}J)\|_q<C+1,\quad \forall k>K.
        \]
        Write: 
        \[
          \pi_{u_k}^*(d_{u_k}J)=\sum_{j=1}^na_{j,k}w_j(\cdot;u_k);
        \]
        since the projections have finite ranges, and all norms are equivalent on finite-dimensional spaces, by the above we deduce that there exists $C'>0$ so that 
        \begin{equation}
          \label{eqn:3}
          \sum_{j,l}a_{j,k}a_{l,k}\langle w_j(\cdot;u_k), w_l(\cdot;u_k)\rangle=\|\pi_{u_k}^*(d_{u_k}J)\|_2^2< C'.
        \end{equation}
        Because of Lemma \ref{lemma:6} and Theorem \ref{thm:unif_conv}  and the fact that $u_k\to \overline u$ weakly in $L^p$, then for every $j=1,\ldots,n$ the function $w_{j}(\cdot;u_k):[0,1]\to \R^d$ converges strongly (and hence in any $L^p$ norm) to a function $\overline w_j:[0,1]\to \R^d$.
        Also, $F(\overline u)=y$ and since $y$ is a regular value, then $\{\overline{w}_1, \ldots, \overline w_n\}$ is a linearly independent set.
        
        By \eqref{eqn:3} we have $\sum_{j,l}a_{j,k}a_{l,k}\langle\overline w_j, \overline w_l\rangle<C'$,
        which tells the sequence:
        \[
        \left \{z_k=\sum_{j}a_{j,k}\overline w_j\right\}_{k\in \mathbb{N}}\subset\textrm{span}\{\overline w_1,\ldots, \overline w_n\}\quad \textrm{is bounded}.
        \]
        Since $\textrm{span}\{\overline w_1,\ldots, \overline w_n\}$ is finite dimensional we can then assume $z_k\to\overline z$;
        since $\{\overline{w}_1, \ldots, \overline w_n\}$ is a linearly independent set then the sequences $\{a_{j,k}\}_{k\in \mathbb{N}}$ for $j=1, \ldots, n$ are bounded and we can assume they converge.  Consequently also $\pi_{u_k}^*(d_{u_k}J)\to \overline z$ (all this up to subsequences).
        
        Finally we have:
        \begin{align*}
        \lim_{k\to \infty} \| d_{u_k}J-\overline z\|_q & \leq \lim_{k\to \infty}\left(\| d_{u_k}J-\pi_{u_k}^*(d_{u_k}J)\|_q\right)\\
        &+\lim_{k\to \infty}\left(\|\pi_{u_k}^*(d_{u_k}J)-\overline z\|_q\right)=0.
        \end{align*}
This proves that $u_k|u_k|^{p-2}=d_{u_k}J\stackrel{L^q}{\longrightarrow}\overline z$ (up to subsequences), and the result follows now from the next Lemma \ref{lemma:grad}.
           \end{proof}
           
\begin{lemma}\label{lemma:grad}Let $\{u_n\}_{n\in \mathbb{N}}\subset L^p$ such that:
\be u_n|u_n|^{p-2}\stackrel{L^q}{\longrightarrow} z.\ee
Then $u_n\stackrel{L^p}{\longrightarrow} z|z|^{(2-p)/(p-1)}.$
\end{lemma}

\begin{proof}Consider the Nemitski operator $N:L^q\to L^p$ defined by $v\mapsto v|v|^{(2-p)/(p-1)}.$ Since:
\be \left| v|v|^{\frac{2-p}{p-1}}\right|\leq |v|^{\frac{1}{p-1}}=|v|^{\frac{p}{p-1}\cdot\frac{1}{p}} \ee
then $N\in C^{0}(L^q, L^p)$ by Theorem \ref{thm:Nemitski}. In particular $u_n=N( u_n|u_n|^{p-2})\stackrel{L^p}{\longrightarrow} N(z)$, and the claim follows.
\end{proof}
\subsection{Critical points}\label{sec:LSappli}
\begin{thm}\label{thm:critical}Let $y$ be a regular value for the endpoint map of the control system \eqref{eq:control}, $1<p<p_c$ (where $p_c$ is given by Theorem \ref{thm:sarychev}) and consider $f=J_p|_{\Omega(y)}$. Then $f$ has infinitely many critical points.
\end{thm}

\begin{proof}The first part of the proof follows the lines of the classical argument.
Assume first that  the fundamental group of $M$ is infinite. Then by \eqref{eq:pik} $\Omega(y)$ has infinitely many components.
Lemma \ref{lemma:C1} tells that $f$ is $C^1$ and Proposition \ref{prop:PS} that it satisfies the Palais-Smale condition. Assume that one component of $\Omega(y)$ does not contain any critical point of $f$. Then we can apply the deformation lemma \cite[Lemma 3.2]{Chang} and conclude that $f$ needs to be unbounded from below, which is in contradiction with the definition $f=J_p|_{\Omega(y)}\geq 0$. 

Assume now the fundamental group of $M$ is finite. Let us call $r:\overline{M}\to M$ the universal covering map. Then $\overline{M}$ is also compact, and the structure $\mathcal{F}$ can be lifted to a structure $\overline{\mathcal{F}}=\{\overline X_0, \ldots, \overline X_d\}$  by setting:
\be d_{\overline{x}}r\overline X_i(\overline{x})=X_i(r(\overline x)).\ee
Let $\overline{x}$ be a lift of  $x$ and  $\{\overline{y}_1, \ldots, \overline{y}_k\}$ be the lifts of $y$ (here $k= \#\pi_1(M)$, the number of sheets of the covering map).  Denote by $\overline{\Omega}$ the set of horizontal curves on $\overline{M}$ leaving from $\overline x$, by $\overline{F}$ the corresponding endpoint map and by $\overline{\Omega}(\overline y)$ the set of horizontal curves on $\overline{M}$ between $\overline{x}$ and $\overline y\in \overline M$. We denote by $\overline{r}:\overline{\Omega}\to \Omega$ the smooth map that associates to a horizontal trajectory $\overline{\gamma}$ on $\overline{M}$ the trajectory $r\circ \overline\gamma$ on $M$. Notice that in coordinates this map is the identity maps on controls (hence it is a local diffeomorphism), and in particular:
\be J(\overline \gamma)=J(\overline{r}(\overline\gamma)).\ee
Moreover, by construction the following diagram is commutative:
$$\begin{tikzpicture}[xscale=2.5, yscale=2]

    \node (A0_0) at (0, 0) {$\overline{M}$};
    \node (A1_0) at (1, 0) {$M$};
    \node (A0_1) at (0, 1) {$\overline{\Omega}(\overline y)$};
    \node (A1_1) at (1, 1) {$\Omega(y) $};
    \path (A0_0) edge [->] node [auto] {$r$} (A1_0);
    \path (A0_1) edge [->] node [auto,swap ] {$\overline{F}$} (A0_0);
    \path (A0_1) edge [->] node [auto] {$\overline{r}$} (A1_1);
    \path (A1_1) edge [->] node [auto] {$F$} (A1_0);
      \end{tikzpicture}
$$
and since $r$ and $\overline{r}$ are local diffeomorphism, then $\overline{y}$ is a regular value of $\overline{F}$.

If we prove the statement for $\overline{M}$, then we are done: in fact given a critical point $\overline{u}$ for the geometric cost $\overline{f}=J|_{\overline{\Omega}(\overline y)}$ then $\overline r (\overline u)$ is a critical point for $f$ (hence we would obtain an infinite numbers of distinct critical points for $f$). To see this fact let us use the Lagrange multiplier formulation: $\overline{u}$ is a critical point of $\overline f$ if and only if there exists $\overline {\lambda}\in T^*_{\overline y}\overline{M}$ such that:
\be\label{eq:critical} \overline{\lambda} \circ d_{\overline{u}}\overline{F}=d_{\overline{u}} J.\ee
Using the commutativity of the above diagram, and the fact that $r$ is a local diffeomorphism we see that this implies the existence of a $\lambda\in T^*_yM$ such that
\be\label{eq:lambda} \lambda \circ d_{\overline{r}(\overline{u})}F\circ d_{\overline{u}}\overline{r}=d_{\overline{r}(\overline{u})}J\circ d_{\overline{u}}\overline{r}:\ee
in fact
\begin{align}
  \label{eqn:critical_2}
  d_{\overline{r}(\overline{u})}J\circ d_{\overline{u}}\overline{r}&=d_{\overline{u}}J\\
  &=\overline{\lambda}\circ d_{\overline{u}}\overline{F}\\
  &=\overline{\lambda}\circ d_{r(\overline{F}(\overline{u}))}r^{-1}\circ d_{\overline{F}(\overline{u})}r\circ d_{\overline{u}}\overline{F}\\
  &=\lambda\circ d_{\overline{r}(\overline{u})}F\circ d_{\overline{u}}\overline r.
\end{align}

On the other hand, being $\overline{r}$ a local diffeomorphism, $d_{\overline{u}}\overline{r}$ is also an isomorphism of vector spaces; consequently simplifying it from \eqref{eq:lambda} we can write:
\be \lambda \circ d_{\overline{r}(\overline{u})}F=d_{\overline{r}(\overline{u})}J\ee
which tells exactly that $\overline{r}(\overline{u})$ is a critical point for $f$.

We are left with the case $M$ compact and \emph{simply connected}. Let $y$ be a regular value of the endpoint map and consider the horizontal path space $\Omega(y)$ endowed with the $W^{1,p}$ topology (recall that we are assuming $1<p<p_c$ with $p_c$ given by Theorem \ref{thm:sarychev}). Since $y$ is a regular value of the Endpoint map, $\Omega(y)$ is a smooth Banach manifold modeled on $L^p=L^{p}([0, 1], \mathbb{R}^d)$ (here $d$ is the rank of the distribution). The function $f$ is $C^1$ (by Lemma \ref{lemma:C1}) and it  satisfies the Palais-Smale condition (by Proposition \ref{prop:PS} above), hence the results follows from Corollary \ref{cor:lust} and the following Proposition.
\begin{prop}[Corollary 3.4 from \cite{Chang}]Let $\Omega(y)$ be Banach manifold and $f\in C^1(\Omega(y), \mathbb{R})$ bounded from below and satisfying the Palais-Smale condition. Then $f$ has at least as many critical points as the Lusternik-Schnirelmann category of $\Omega(y)$.
\end{prop}\end{proof}
\section{The subriemannian case}\label{sec:subriemannian}

In this section we discuss applications of the previous results to the subriemannian case, in particular we will always make the assumption $X_0=0$.

\subsection{Geodesics}

Given two points $x,y$ in a subriemannian manifold $M$, a \emph{subriemannian geodesic} is a curve $\gamma:I\to M$ satisfying the following properties: (i) it is absolutely continuous; (ii) its derivative (which exists almost everywhere) belongs to the subriemannian distribution; (iii) it is parametrized by arc-length; (iv) $\gamma(0)=x$ and $y(1)=y$; (v) it is locally length minimizer, i.e. for every $t\in [0,1]$ there exists $\delta(t)>0$ such that $\gamma|_{[t-\delta(t), t+\delta(t)]}$ has minimal length among all horizontal curves joining $\gamma(t-\delta(t))$ with $\gamma(t+\delta(t)).$
\begin{prop}\label{prop:geodesics}Let $y$ be a regular value of the Endpoint map centered at $x$. For every $p>1$ all critical points of $f=J_p\big|_{\Omega(y)}$ are subriemannian geodesics joining $x$ to $y$.
\end{prop}
\begin{proof}First let us notice that curves that are \emph{locally} $J_p$-minimizers are parametrized by constant speed and are locally length minimizer (the proof of this fact is the same as the classical proof for $p=2$ as in \cite[Section 12]{Milnor} and essentially uses the fact that $\left( \int |u|\right)^p\leq \int |u|^p$ with equality if and only if $|u|\equiv c$).
Also, being locally length minimizer and parametrized by constant speed implies that \emph{globally} the parametrization is with constant speed. 

Let us consider the equation for $u\in L^p$ to be a critical point of $f=J_p|_{F^{-1}(y)}$ (using Lagrange multipliers rule):
\be\label{eq:critp}\exists \lambda\in T^*_yM\quad \textrm{such that}\quad  \lambda \circ d_uF=pu|u|^{p-2}.
\ee
In particular since a critical point $u$ of $f$ is a \emph{local} length minimizer (this can be seen by considering variations of only a small portion of the corresponding curve), we must have $|u|\equiv c>0$ and we can rewrite \eqref{eq:critp} as:
\be \exists \eta=\frac{\lambda}{p c} \in T^*_yM\quad \textrm{such that}\quad  \eta \circ d_uF=u,\ee
which is the equation for the critical points of $J_2$ on $\Omega(y)$. 

Thus if $y$ is a regular value of the Endpoint map, the critical points of $J_2$ and $J_p$ on $\Omega(y)$ are the same; since critical points of $J_2\big|_{\Omega(y)}$ are subriemannian geodesics joining $x$ to $y$ (see \cite[Theorem 4.57]{AgrachevBarilariBoscain}), the result follows.
\end{proof}
As a corollary of Propositon \ref{prop:geodesics} and Theorem \ref{thm:critical}, we obtain the subriemannian version of Serre's theorem. 
\begin{thm}[Subriemannian Serre's Theorem]If $y$ is a regular value of the endpoint map centered at a point $x$ in a compact \emph{subriemannian} manifold, the set of subriemannian geodesics joining $x$ and $y$ is infinite.
\end{thm}
\subsection{The contact case}
In the contact case we can remove from the subriemannian Serre's theorem the \emph{regularity} assumption on the two points. In fact the same proof works in the slightly more general case of \emph{fat} distributions (see \cite{Montgomery} for more details on these distributions), as the only property that we are going to use is that there are no nontrivial abnormal curves.

\begin{thm}\label{thm:contact}For every two points on a compact, contact subriemannian manifold the set of subriemannian geodesics joining them is infinite.\end{thm}
\begin{proof}We prove that $J_p$ (with $p>1$) has infinitely many critical points when restricted to each $\Omega(y)$. Because of Theorem \ref{thm:critical} the only case that we have to cover is the case the final point $y$ is the same point as the initial point $x$ (in which case it is not a regular value for $F$). 

Recall that on a contact manifold there are no \emph{nontrivial} abnormal extremals (i.e. critical points of the Endpoint map), see \cite[Corollary 4.3.5]{AgrachevBarilariBoscain}, the trivial one being the one with zero control.

The case when the base manifold is not simply connected can be treated as in the proof of Theorem \ref{thm:critical}: if the fundamental group is infinite, then only one of the infinitely many components of $\Omega(x)$ contains the zero control; if the fundamental group is finite, we pass to the universal cover (which is still compact) and notice that the projection of a geodesic is still a geodesic (no matter if it is a singular point of the Endpoint map, as in the subriemannian case geodesics are locally length minimizers and length is preserved by projection). 

Thus we assume our manifold $M$ is compact and simply connected. Consider $\tilde{F}$, the restriction to $L^p\backslash\{0\}$ of the Endpoint map centered at $x$. Then, again by \cite[Corollary 4.3.5]{AgrachevBarilariBoscain}, $\tilde{F}^{-1}(x)$ is a smooth Banach manifold and:
\be \Omega(x)=\tilde{F}^{-1}(x)\cup\{0\}\ee
($ \Omega(x)$ has its only singularity at zero). 

We prove that the Lusternik-Schnirelmann category of $\tilde{F}^{-1}(x)$ is infinite. Combining this with the fact that the $p$-Energy $f:\tilde{F}^{-1}(x)\to \R$ is $C^1$ and satisfies Palais-Smale for every level $c>0$, implies that $f$ has infinitely many critical points.

Assume that the Lusternik-Schnirelmann category of $\tilde{F}^{-1}(x)$ is finite and let $U_1,\ldots, U_k$ be contractible open sets covering $\tilde{F}^{-1}(x)$. We show that a sufficiently small neighborhood $U_0$ of $0\in \Omega(x)$ is contractible in $\Omega(x)$: this would imply that the Lusternik-Schnirelmann category of $\Omega(x)$ is finite as well, which contradicts Corollary \ref{cor:lust} (all spaces $\Omega(y)$, regardless $y$, are homotopy equivalent since the endpoint map is a Hurewicz fibration hence they all have the same L-S category). Let now $c>0$  be such that $f$ has no critical values in $(0,c)$. Notice that if there is a sequence $\{c_n\}_{n\in \mathbb{N}}$ of critical values of $f$ converging to zero, then we immediately have infinitely many critical points so we can assume such a sequence does not exist, which implies the existence of $c>0$. Consider the open set $U_0=\Omega(y)\cap\{f<c\}$ and a pseudo-gradient vector field $\tilde X\in \textrm{Vec}(\tilde{F}^{-1}(x))$ for $f$ (the existence of such a vector field is guaranteed by \cite[Theorem 3.1]{Chang}). By the deformation Lemma \cite[Lemma 3.2]{Chang} the (semi)-flow of $\tilde{X}$ (we extend this flow to a constant map at zero) deformation retracts $U_0$ to $0\in \Omega(x).$ 
\end{proof}

\subsection{The subriemannian loop space has the homotopy type of a CW-complex}
\begin{thm}\label{thm:CW}Assume the endpoint map for a subriemannian manifold has \emph{at least} one regular value. Then the space $\Omega(y)$ with the $W^{1,2}$ topology has the homotopy type of a CW-complex.
\end{thm}
\begin{proof}
Every Hilbert manifold has the homotopy type of a countable CW-complex \cite{Palaisinfinite}. In particular this is true for the space $\Omega(y')$ if $y'$ is a regular value of the Endpoint map. All spaces $\Omega(y)$ as $y$ varies on $M$ have the same homotopy type (as a consequence of Theorem \ref{thm:sarychev}), hence the result follows.
\end{proof}
\begin{cor}\label{cor:deformation}If the endpoint map for a subriemannian manifolds has at least one regular value, the inclusion of the horizontal path space in the ordinary one is a homotopy equivalence (both spaces endowed with the $W^{1,2}$ topologies). 
\end{cor}
\begin{proof}The inclusion of the horizontal path space in the ordinary one is a \emph{weak} homotopy equivalence for the $W^{1,2}$ topologies (by Theorem \ref{thm:sarychev}). Since both spaces have the homotopy type of CW-complexes (by Theorem \ref{thm:CW}) the result follows from Whitehead's theorem \cite[Theorem 4.5]{Hatcher}.
\end{proof}
\section{Appendix}
In this section we collect a list of technical results that we use in the proofs. Most of these results are well known to experts, but it is often not easy to find an appropriate reference. Some proofs are adaptations from \cite{Trelat} to the general case $p\in (1, \infty)$. 
\begin{lemma}[Gronwall inequality]
		Assume $\varphi:[0,T]\to\R$ to be a bounded nonnegative measurable function, $\alpha:[0,T]\to\R$ to be a nonnegative integrable function and $B:[0,T]\to\R$ to be non decreasing such that
		\be \varphi(t)\leq B(t)+\int_0^t\alpha(\tau)\varphi(\tau)d\tau,\quad\forall t\in[0,T];\ee
		then 
		\be\varphi(t)\leq B(t)e^{\int_0^t\alpha(\tau)d\tau},\quad\forall t\in[0,T].\ee
	\end{lemma}
	
	\begin{prop}\label{prop:domain}
		Let $T>0$ be fixed. Then the domain of the endpoint map is open in $L^p([0,T],\R^d)$.
	\end{prop}
	
	\begin{proof}
		The strategy of the proof consists in showing that if $v$ belongs to a sufficiently small neighborhood of $u$ in $L^p([0,T],\R^d)$, then the corresponding trajectories $\gamma_u$ and $\gamma_v$ remain uniformly close. It is not restrictive to prove the theorem for small $T>0$, which in turn allows us to work inside a coordinate chart. Also, we assume that the vector fields $X_i$, $i=0,1,\dotso, d$ have compact support in $\R^n$; Lemma $3.2$ in the aforementioned paper yields that they are therefore globally Lipschitzian. For any $t\in[0,T]$ we have the following:
		\begin{alignat}{9}
		   \|\gamma_u(t)-\gamma_v(t)\|&\leq\bigg\|&&\int_0^t(X_0(\gamma_u(\tau))-X_0(\gamma_v(\tau)))d\tau\\
		                              &&&+\int_0^t\sum_{i=1}^dv_i(\tau)(X_i(\gamma_u(\tau))-X_i(\gamma_v(\tau)))d\tau\\
		                              &&&-\int_0^t\sum_{i=1}^d(v_i(\tau)-u_i(\tau))X_i(\gamma_u(\tau))d\tau\bigg\|\\
		                              &\leq C&&\int_0^t(1+\sum_{i=1}^d|u_i(\tau)|)\|\gamma_u(\tau)-\gamma_v(\tau)\|d\tau+h_v(t),
		\end{alignat}
		with 
		\be h_v(t)=\left\|\int_0^t\sum_{i=1}^d(v_i(\tau)-u_i(\tau))X_i(\gamma_u(\tau))d\tau\right\|.\ee By H\"older inequality we obtain
		\be h_v(t)\leq C'T^{1/q}\|u-v\|_p,\quad\forall t\in[0,T];\ee
		moreover we deduce that for any $\varepsilon>0$ there exists a neighborhood $U$ of $u$ in $L^p([0,T],\R^d)$ such that $h_v(t)\leq \varepsilon$, for any $v\in U$ and $t\in[0,T]$. We conclude using Gronwall inequality that 
		\be\|\gamma_u(t)-\gamma_u(v)\|\leq \varepsilon e^{C(T+T^{1/q}K)},\quad\forall t\in[0,T].\ee
	\end{proof}	
	
	\begin{thm}
		\label{thm:unif_conv}
		Let $u=(u_1,\dotso,u_d)\in L^p([0,T],\R^d)$ be a control in the domain of the endpoint map $F$, and let $\gamma_u$ be the corresponding solution to \eqref{eqn:1}. Let $(u_n)_{n\in\mathbb N}$ be a sequence in $L^p([0,T],\R^d)$. If $u_n\stackrel{L^p}{\rightharpoonup} u$, then for $n$ large enough $\gamma_{u_n}$ is well-defined on $[0,T]$ and moreover $\gamma_{u_n}$ converges to $\gamma_u$, uniformly on $[0,T]$.
	\end{thm}
	
	\begin{proof}
		It suffices to prove the proposition when $T$ is close to zero; this in turn permits to work in a coordinate chart, that is we may suppose the vector fields $X_i$ to have compact support in $\R^n$. Moreover, let $K$ be a compact neighborhood of $x$ such that there exists $C>0$ for which
		\[
		  \|X_i(z_1)-X_i(z_2)\|\leq C\|z_1-z_2\|
		\]
		holds for any $z_1,z_2\in K$ and any $i=0,1,\dotso,d$. For all $t\in[0,T]$ we have:
		\begin{align}
		  \|\gamma_u(t)-\gamma_{u_n}(t)\|&\leq\int_0^t\|(X_0(\gamma_u(\tau))-X_0(\gamma_{u_n}(\tau))\|d\tau\\
		  &+\int_0^t\sum_{i=1}^d|u_{n,i}(\tau)|\|X_i(\gamma_u(\tau))-X_i(\gamma_{u_n}(\tau))\|d\tau\\
		  &+\int_0^t\sum_{i=1}^d|u_{n,i}(\tau)-u_i(\tau)|\|X_i(\gamma_u(\tau))\|d\tau\\
		  &\leq C\int_0^1(1+\sum_{i=1}^d|u_{n,i}(\tau)|)\|\gamma_u(\tau)-\gamma_{u_n}(\tau)\|d\tau+h_n(t),
		\end{align}
		where 
		\be\label{eqn:2}h_n(t)=\int_0^t\sum_{i=1}^d|u_{n,i}(\tau)-u_i(\tau)|\|X_i(\gamma_u(\tau))\|d\tau.\ee
		The uniform boundedness principle of Banach and Steinhaus ensures that $\sup_{n\in\mathbb N}\|u_n\|_p\leq M$; if we can prove that $h_n$ tends \emph{uniformly} on $[0,T]$ to the zero function, then we would finish the argument using the Gronwall inequality.
		
		Observe that $h_n$ tends pointwise to the zero function; it is also uniformly $1/q$-H\"olderian, where $q=\frac{p}{p-1}$, indeed if $L=\sup_i\sup_{p\in\R^n}\|X_i(p)\|$ we have
		\begin{align}
		\label{eqn:3a}
		\|h_n(t_1)-h_n(t_2)\|&\leq L\int_{t_1}^{t_2}\sum_{i=1}^d(|u_{n,i}(\tau)|+|u_i(\tau)|)d\tau\\
		&\leq L(M+\|u\|_p)|t_1-t_2|^{1/q}.
		\end{align}
		The proof is then concluded by the next lemma.
	\end{proof}
	
	\begin{lemma}[Uniform convergence of H\"olderian maps]
		\label{lemma:unif_conv_hol}
		Let $\{f_k\}_{k\in\mathbb N}:[a,b]\to\R^n$ be a uniformly $\alpha$-H\"olderian sequence of functions which converges pointwise to a limit function $f$. Then $f$ is $\alpha$-H\"olderian and $f_k\to f$ uniformly on $[a,b]$.
	\end{lemma}
	
	\begin{proof}
		The relation $\|f_k(x)-f_k(y)\|\leq M|x-y|^\alpha$ immediately yields that the limit function $f$ is also $\alpha$-H\"olderian.
		
		Next, let $\varepsilon>0$ be arbitrary and let accordingly $\rho=\left(\frac{\varepsilon}{3M}\right)^{1/\alpha}$. As $[a,b]$ is compact, it can be covered by a finite collection $\{B_i\}_{i=1}^l$ of balls of radius $\rho$, whose centers will be denoted by $x_i$; this means that for any $x\in[a,b]$ there exists $i\in\{1,\dotso,l\}$ such that $|x-x_i|\leq \rho$. Let $K\in\mathbb N$ be such that $\|f_k(x_i)-f(x_i)\|\leq \varepsilon/3$ for all $i=1,\dotso,l$ if $k>K$. The following holds true for $k\in\mathbb N$ sufficiently large:
		\begin{align}
		   \label{eqn:4}
		   \|f_k(x)-f(x)\|&\leq \|f_k(x)-f_k(x_i)\|+\|f_k(x_i)-f(x_i)\|+\|f(x_i)-f(x)\|\\
		   &\leq 2M|x-x_i|^\alpha+\frac{\varepsilon}{3}\leq \varepsilon,
		\end{align}
		and this finishes the proof.
	\end{proof}
	
	We turn now to the issue of the differentiability of the endpoint map $F$, i.e. we want to determine its Fr\'echet differential and prove some of its continuity properties.
	
	\begin{prop}
		Let $u$ be in the domain of the endpoint map $F:L^p([0,T],\R^d)$ and let $\gamma_u$ be the associated trajectory. Then for any bounded neighborhood $U$ of $u$ in $L^p([0,T],\R^d)$, there exists a constant $C=C(U)$ such that whenever $v,w\in U$ and $t\in[0,T]$ we have
		\be\|\gamma_v(t)-\gamma_w(t)\|\leq C\|v-w\|_p.\ee
	\end{prop}
	
	\begin{proof}
		Using \eqref{eqn:1} we derive the following estimate
		\begin{align}
		 \label{eqn:5}
		 \|\gamma_v(t)-\gamma_w(t)\|&\leq \sum_{i=1}^d \int_0^t|v_i-w_i|\|X_i(\gamma_v(s))\|ds+\int_0^t\|X_0(\gamma_v(s))-X_0(\gamma_w(s))\|ds\\
		 &+\sum_{i=1}^d\int_0^t|w_i|\|X_i(\gamma_v(s))-X_i(\gamma_w(s))\|ds.
		\end{align}
		Theorem \ref{thm:unif_conv} ensures that $\gamma_v$ and $\gamma_w$ take values in a compact $K$ which depends just on $U$\footnote{By Banach Alaoglu $U$ is sequentially weakly compact, hence weakly compact by the Eberlein Smulian theorem. On the other hand theorem \ref{thm:unif_conv} implies that for any $\varepsilon>0$, whenever $u,v$ belong to a sufficiently small open set, $\|\gamma_u(t)-\gamma_v(t)\|\leq \varepsilon$ on $[0,T]$. The statement follows since whenever we cover $U$ with a collection of open sets of arbitrary small size, we may always extract a finite subcover and then proceed via the triangular inequality.}; as $X_0,X_1,\dotso,X_d$ are smooth, we have the existence of a constant $M$ such that for all $v,w\in U$ and for all $i\in 1,\dotso,d$ there holds
		\begin{align}
		  \|X_i(\gamma_v)\|&\leq M,\\
		  \|X_i(\gamma_v)-X_i(\gamma_w)\|&\leq M\|\gamma_v-\gamma_w\|,\quad\forall t\in[0,T];
		\end{align}
		lastly we may assume that $U$ is contained in a ball of radius $R$, that is $\|w\|_p\leq R$ for all $w\in U$. We proceed with the estimate in \eqref{eqn:5} as
		\be
		  \label{eqn:6}
		  \|\gamma_v(t)-\gamma_w(t)\|\leq B\|v-w\|_p+M\int_0^t(1+\sum_{i=1}^d|w_i|)\|\gamma_v(s)-\gamma_w(s)\|ds,\quad\forall t\in[0,T],
		\ee
		where $B=MT^{1/q}$; finally, Gronwall inequality yields
		\be\label{eqn:7}\|\gamma_v(t)-\gamma_w(t)\|\leq Be^{M(T+RT^{1/q})}\|v-w\|_p,\quad\forall t\in[0,T].\ee
	\end{proof}
	
	We fix now some notations used in the next theorem: let $A_u(t)=dX_0(\gamma_u)+\sum_{i=1}^du_idX_i(\gamma_u)$, $B_u(t)=(X_1(\gamma_u),\dotso,X_d(\gamma_u))$, and let $M_u$ be the $n\times n$ matrix solution of $M_u'=A_uM_u$ satisfying $M_u(0)=I$; we have
	
	\begin{thm}[Differentiability of the endpoint map]
		\label{thm:differentiability}
		The endpoint map $F$ is $L^p$-Fr\'echet differentiable; its differential at $u$ is the linear map $dF(u):L^p\to\R^n$ defined by
		\be (d_uF)v=\int_0^TM_u(T)M_u(s)^{-1}B_u(s)v(s)ds.\ee
	\end{thm}
	
	\begin{proof}
		Let $u\in L^p([0,T],\R^d)$ be fixed in the domain of $F$. Let us consider a neighborhood $U$ of $u$ in $L^p$; without loss of generality we may assume that there exists $R>0$ such that $\|v\|_p\leq R$ for any $v\in U$. Let $\gamma_u$ and $\gamma_{u+v}$ be the solutions to \eqref{eqn:1} with respect to the controls $u$ and $u+v$ respectively. We have
		\be\label{eqn:8}\dot\gamma_{u+v}-\dot\gamma_u=X_0(\gamma_{u+v})-X_0(\gamma_u)+\sum_{i=1}^dv_iX_i(\gamma_{u+v})+\sum_{i=1}^du_i(X_i(\gamma_{u+v})-X_i(\gamma_u)).\ee
		For all $i=0,1,\dotso,d$ there hold the expansions
		\begin{align}
		  X_i(\gamma_{u+v})-X_i(\gamma_u)&=dX_i(\gamma_u)(\gamma_{u+v}-\gamma_u)\\&+\int_0^1(1-t)d^2X_i(t\gamma_u+(1-t)\gamma_{u+v})(\gamma_{u+v}-\gamma_u,\gamma_{u+v}-\gamma_u)dt,\\X_i(\gamma_{u+v})&=X_i(\gamma_u)+\int_0^1(1-t)dX_i(t\gamma_u+(1-t)\gamma_{u+v})(\gamma_{u+v}-\gamma_u)dt;
		\end{align}
		plug the above into \eqref{eqn:8} to rewrite that equation as \be\label{eqn:9}\dot\omega=A_u\omega+B_uv+\xi,\ee
		where $\omega(t)=\gamma_{u+v}(t)-\gamma_u(t)$ and
		\begin{align}
		\xi(t)&=\sum_{i=1}^dv_i(t)\int_0^1(1-s)dX_i(s\gamma_u+(1-s)\gamma_{u+v})(\gamma_{u+v}-\gamma_u)ds\\&+\int_0^1(1-s)d^2X_0(s\gamma_u+(1-s)\gamma_{u+v})(\gamma_{u+v}-\gamma_u,\gamma_{u+v}-\gamma_u)ds\\&+\sum_{i=1}^du_i(t)\int_0^1(1-s)d^2X_i(s\gamma_u+(1-s)\gamma_{u+v})(\gamma_{u+v}-\gamma_u,\gamma_{u+v}-\gamma_u)ds.
		\end{align}
		We have $\|v\|_p\leq R$ for all $v\in U$; the previous proposition and the estimate
		\be\|s\gamma_u(s)+(1-s)\gamma_{u+v}(s)\|\leq \|\gamma_u(s)\|+(1-s)\|\gamma_{u+v}(s)-\gamma_u(s)\|\leq \|\gamma_u(s)\|+CR\ee
		imply that there exists a compact $K\subset\R^n$ such that $s\gamma_u(s)+(1-s)\gamma_{u+v}(s)\in K$ for any $s\in[0,1]$ and any $v\in U$. Since the $X_i$ are smooth, again by the proposition above we have we the estimate
		\be\|\xi(t)\|\leq c_1\|v\|_p\sum_{i=1}^d|v_i(t)|+c_2\|v\|_p^2(1+\sum_{i=1}^d|u_i(t)|).\ee
		We solve \eqref{eqn:9} to obtain
		\be\omega(t)=\int_0^tM_u(t)M_u(s)^{-1}B_u(s)v(s)ds+\int_0^tM_u(t)M_u(s)^{-1}\xi(s)ds;\ee
		in particular for $t=T$
		\begin{align}
		  \label{eqn:10}
		  \bigg\|\gamma_{u+v}(T)-\gamma_u(T)&-\int_0^TM_u(T)M_u(s)^{-1}(s)B_u(s)v(s)ds\bigg\|\\&\leq C'\left(c_1\|v\|_p\int_0^T\sum_{i=1}^d|v_i(s)|ds+c_2\|v\|_p^2\int_0^T(1+\sum_{i=1}^d|u_i(s)|)ds\right)\\&\leq C'\left(c_1T^{1/q}+c_2(T+\|u\|_pT^{1/q})\right)\|v\|_p^2.
		\end{align}
		The map
		\be \mathcal F_u:L^p\ni v\mapsto\int_0^TM_u(T)M_u(s)^{-1}B_u(s)v(s)ds\in\R^n\ee
		is evidently linear and by \eqref{eqn:10} also continuous. It then follows that the endpoint map $F$ is differentiable at $u$ and $d_uFu=\mathcal F_u$.
	\end{proof}
	
	\begin{thm}
		\label{thm:unif_conv_diff}
		Let $u=(u_1,\dotso,u_d)\in L^p([0,T],\R^d)$ be a control in the domain of the endpoint map $F$. Let $(u_n)_{n\in\mathbb N}$ be a sequence in $L^p([0,T],\R^d)$ such that $u_n\stackrel{L^p}{\rightharpoonup} u$ for some $u\in L^p([0,T],\R^d)$. Then $d_{u_n}F\to d_uF$.
	\end{thm}
	
	The proof of this theorem needs a series of preliminary lemmas; for $s\in[0,T]$, set $N_u(s)=M_u(T)M_u(s)^{-1}$. Since $N_u(s)M_u(s)=M_u(T)$, upon differentiation and using the definition of $M_u$, we obtain $N'_u(s)M_u(s)+N_u(s)A_u(s)M_u(s)=0$, that is
	\be\label{eqn:11}N'_u(s)=-N_u(s)A_u(s),\quad N_u(T)=I.\ee
	
	\begin{lemma}
		\label{lemma:6}
		Let $\{u_n\}_{n\in\mathbb N}$ and $u$ be as in the statement of theorem \ref{thm:unif_conv_diff}. Then $N_{u_n}\to N_u$ uniformly on $[0,T]$.
	\end{lemma}
	
	\begin{proof}
		\begin{alignat}{9}
		N_u(t)-N_{u_n}(t)&=\int_0^t\bigg(&& N_{u_n}(s)(dX_0(\gamma_{u_n}(s))+\sum_{i=1}^du_{n,i}dX_i(\gamma_{u_n}(s)))\\		 \label{eqn:bla}
&&&-N_u(s)(dX_0(\gamma_u(s))+\sum_{i=1}^du_i(s)dX_i(\gamma_u(s)))\bigg)ds\\&=\int_0^t\bigg(&&( N_{u_n}(s)-N_u(s))dX_0(\gamma_u(s))+N_u(s)(dX_0(\gamma_{u_n}(s))-dX_0(\gamma_u(s)))
		 \\&&&+(N_{u_n}(s)-N_u(s))\sum_{i=1}^du_{n,i}(s)dX_i(\gamma_{u_n}(s))\\&&&+N_u(s)\sum_{i=1}^du_{n,i}(s)(dX_i(\gamma_{u_n}(s))-dX_i(\gamma_u(s)))\\&&&+N_u(s)\sum_{i=1}^d(u_{n,i}(s)-u_i(s))dX_i(\gamma_u(s))\bigg)ds.
		\end{alignat}
		By virtue of theorem \ref{thm:unif_conv}, $\gamma_{u_n}\to\gamma_u$ uniformly on $[0,T]$; moreover if
		\be h_n(t)=\int_0^1N_u(s)\sum_{i=1}^d(u_{n,i}(s)-u_i(s))dX_i(\gamma_u(s))ds,\ee
		then $\|h_n\|\to 0$ uniformly on $[0,T]$ by lemma \ref{lemma:unif_conv_hol}: indeed the sequence $\{h_n\}_{n\in\mathbb N}$ is $1/q$- H\"olderian and converges pointwise to $0$, moreover the factor $N_u(s)$ does not depend on $n$. Then \eqref{eqn:bla} can be estimated for $n$ sufficiently large as 
		\be\|N_u(t)-N_{u_n}(t)\|\leq C\int_0^t\|N_u(s)-N_{u_n}(s)\|ds+\varepsilon,\ee and the theorem follows using the Gronwall inequality, as desired. 
	\end{proof}
	
	\begin{proof}[Proof of theorem \ref{thm:unif_conv_diff}]
		 Theorem \ref{thm:differentiability} yields that the differential of the endpoint map at the point $w$ has the form
		\be (d_wF)v=\int_0^TN_w(s)B_w(s)v(s)ds.\ee 
		We know from theorem \ref{thm:unif_conv} that $\gamma_{u_n}\to\gamma_u$ uniformly on $[0,T]$; then $B_{u_n}\to B_u$ uniformly on $[0,T]$. As lemma \ref{lemma:6} shows that also $N_{u_n}\to N_u$ uniformly on $[0,T]$, we deduce that 
		\be d_{u_n}Fv\to d_uFv\ee
		uniformly on $[0,T]$, for any $v\in L^p([0,T],\R^d)$, and this finishes the proof.
    \end{proof}



\begin{thebibliography}{9}

\bibitem{Agrachevsmoothness}A. A. Agrachev: \emph{Any sub-Riemannian metric has points of smoothness}, Russian Math. Dokl., 2009, v.79, 1--3 
 \bibitem{AgrachevBarilariBoscain} A.~A.~Agrachev, U.~Boscain, D.~Barilari: \emph{Introduction to Riemannian and sub-Riemannian geometry},     lectures notes, available at: http://www.cmapx.polytechnique.fr/~barilari/Notes.php         

\bibitem{AmbrosettiProdi}A. Ambrosetti, G. Prodi:\emph{A primer in nonlinear analysis}, Cambridge University Press 1993



\bibitem{AgrachevGentileLerario} A. A. Agrachev, A. Gentile, A. Lerario: \emph{Geodesics and horizontal-paths paces in Carnot groups}, Geometry \& Topology, to appear
\bibitem{AgrachevLee}A. A. Agrachev, P.Lee: \emph{Continuity of optimal control cost and its application to weak KAM theory}, Calculus of Variations and Part. Dif. Eq., 2010, v.39, 213--232
\bibitem{agrasary}A. A. Agrachev, A. V. Sarychev: \emph{Abnormal sub-Riemannian geodesics: Morse index and rigidity}, Ann. Inst. H. Poincar\'e Anal. Non Lin\'eaire 13(6) (1996) 635–690
\bibitem{BarilariLerario}D. Barilari, A. Lerario: \emph{Geometry of Maslov Cycles},
Geometric Control and Sub-Riemannian Geometry Springer INdAM series (2014)
\bibitem{BottTu} R. Bott and L. Tu: \emph{Differential Forms in Algebraic Topology}, Springer-Verlag, 1982.
\bibitem{Bryant} R. L. Bryant and L. Hsu: \emph{Rigidity of integral curves of rank 2 distributions}, Invent. Math. 114(2) (1993) 435–461.
    	\bibitem{Carothers} N.L. Carothers, \emph{A short course on Banach space theory}

\bibitem{Chang} K.~Chang: \emph{Infinite dimensional Morse Theory and Multiple Solution Problems}, Birkh\"auser, 1993
%

\bibitem{dynamic}J. Dominy, H. Rabitz: \emph{Dynamic Homotopy and Landscape Dynamical Set Topology in
Quantum Control}
      \bibitem{Hatcher} A. Hatcher: \emph{Algebraic Topology}, Cambridge University Press, 2002.
\bibitem{Hurewicz} W. Hurewicz: \emph{On the concept of fiber space}, Proc Natl Acad Sci U S A. 1955 Nov 15; 41(11): 956–961.
\bibitem{James}I.M. James: \emph{On category, in the sense of Lusternik-Schnirelmann}, Topology, Volume 17, Issue 4, 1978, Pages 331–348
\bibitem{Jean}F. Jean: \emph{Control of Nonholonomic Systems and Sub-Riemannian Geometry}, 	arXiv:1209.4387
\bibitem{LerarioRizzi}A. Lerario and L. Rizzi: \emph{How many geodesics joint two points on a contact subriemannian manifold?}, 
\bibitem{Milnor} J. W. Milnor: \emph{Morse theory}, Princeton University Press
\bibitem{Montgomery}R. Montgomery: \emph{A Tour of Subriemannian Geometries, Their Geodesics and Applications}, AMS Mathematical Surveys and Monographs.
\bibitem{Palaisinfinite}R. S. Palais: \emph{Homotopy theory of infinite dimensional manifolds}, Topology 5 (1966)
\bibitem{Sarychev}A. V. Sarychev: \emph{On homotopy properties of the space of trajectories of a completely
nonholonomic differential system}, Soviet Math. Dokl. 42, 674-678 (1991).
\bibitem{Schwartz} J. T. Schwartz: \emph{Generalizing the Lusternik-Schnirelmann Theory of Critical Points}, Communications on pure and applied mathematics, Vol XVII, 307-315 (1964)
\bibitem{Serre}J-P. Serre: \emph{Homologie Singuliere Des Espaces Fibres} Annals of Mathematics Second Series, Vol. 54, No. 3 (Nov., 1951), pp. 425-505
\bibitem{Smale1}S. Smale: \emph{Regular curves on Riemannian manifolds}, Trans. Amer. Math. Soc. 87,
492–512 (1958)
\bibitem{Spanier}H. Spanier: \emph{Algebraic Topology}, Springer
\bibitem{Trelat}E. Tr\'elat: \emph{Some properties of the value function and its level sets for affine control system with quadratic cost}, Journal of Dynamical and Control Systems
October 2000, Volume 6, Issue 4, pp 511-541
\end{thebibliography}
\end{document}